\documentclass[reqno]{amsart}
\usepackage[foot]{amsaddr}

\usepackage[utf8]{inputenc}
\usepackage{xcolor}
\usepackage[charsperline=100]{jlcode}

\usepackage[hidelinks]{hyperref}
\usepackage{url}
\usepackage{mathtools}
\usepackage{amssymb}
\usepackage{amscd}
\usepackage{amsthm}
\usepackage{setspace}
\usepackage{enumerate}
\usepackage{graphicx}
\usepackage{mathtools}
\usepackage{tcolorbox} 
\usepackage{subfig}
\usepackage{float}
\usepackage{siunitx}
\usepackage{tikz-cd}
\usepackage{bm,blkarray}
\numberwithin{equation}{section}

\usepackage{algorithm}
\usepackage{algpseudocode}

\usepackage{mathtools}
\usepackage[tableposition=top]{caption}
\usepackage{booktabs,dcolumn}

\usepackage[obeyFinal,textsize=footnotesize]{todonotes}

\renewcommand\H{\mathbb{H}}
\newcommand\R{\mathbb{R}}
\newcommand\Z{\mathbb{Z}}

\renewcommand\r{\mathbf{r}}

\newcommand{\norm}[1]{\left\Vert #1 \right\Vert}
\newcommand{\set}[1]{\left\{#1\right\}}

\renewcommand{\H}{\mathcal H}
\renewcommand{\S}{\mathcal{S}}

\title[The largest 5\textsuperscript{th} pivot may be the root of a 61\textsuperscript{st} degree polynomial]{The largest 5\textsuperscript{th} pivot may be the root of a 61\textsuperscript{st} degree polynomial}
\author{James Chen}
\address{Department of Mathematics, Massachusetts Institute of Technology, \newline \indent Cambridge, MA, 02139 USA.}
\email{jchen163@mit.edu}
\author{Alan Edelman}
\email{edelman@mit.edu}
\author{John Urschel}
\email{urschel@mit.edu}
\subjclass[2020]{Primary 65F05, 15A23.}
\keywords{Gaussian elimination, growth factor, complete pivoting.}

\newtheorem{theorem}{Theorem}[section]

\newtheorem{lemma}[theorem]{Lemma}
\newtheorem{conjecture}[theorem]{Conjecture}
\newtheorem{proposition}[theorem]{Proposition}

\newtheorem{remark}[theorem]{Remark}

\begin{document}

\newenvironment{dedication}
{\textbf{Dedication:}\itshape}

\begin{abstract}
This paper introduces a number of new techniques in the study of the famous question from numerical linear algebra: what is the largest possible growth factor when performing Gaussian elimination with complete pivoting?
This question is highly complex, due to
a complicated set of polynomial inequalities that need to be simultaneously satisfied. This paper introduces
the  JuMP + Groebner basis + discriminant polynomial approach as well as the use of interval arithmetic computations.  Thus, we are introducing a marriage of
numerical and exact mathematical computations.

In 1988, Day and Peterson performed numerical optimization on $n=5$
with NPSOL and obtained a largest seen value of $4.1325...$.  This same best value was reproduced by Gould
with LANCELOT in 1991.  We ran extensive comparable experiments with the modern software tool
JuMP and also saw the same value $4.1325...$.  While the combinatorial explosion
of possibilities prevents us from knowing whether there may not be a larger maximum,
we succeed in obtaining the exact mathematical value: 
the number  $4.1325...$
is exactly the root of a 61st degree polynomial provided in this work,
and is a maximum given the equality constraints seen by JuMP.
In light of the numerics, we pose the conjecture that this lower bound is indeed the maximum. We also apply this technique to $n = 6$, $7$, and $8$.

Furthermore, in 1969, an upper bound of $4\frac{17}{18}\approx 4.94$ was produced for the maximum
possible growth for $n = 5$. We slightly lower this upper bound to $4.84$.

\end{abstract}

\maketitle

\noindent\begin{dedication}
\textit{This paper is dedicated to the memory of Nick Higham, who was fascinated by Gaussian elimination \cite{higham2021random,higham1989accurate,higham1990bounding,higham1990stability,higham1993optimization,higham2002accuracy,higham2011gaussian,higham1989large},  and the conjecture and the techniques in this paper. We are saddened that we could not, at the time, tell him all the details. We miss him greatly. He was a good friend and a great numerical linear algebraist.}
\end{dedication}

\section{The maximum growth factor sequence for $n=1,\ldots,5$}

What number might be in the $n=5$ position  in the sequence $1, 2, 2\tfrac{1}{4},4$?
This sequence, known since the 1960s, consists of the largest growth factors for complete pivoting. Growth ``chasers'' have long attempted to nail down the value for the next dimension, $n=5$. The value $4.1325\ldots$ was produced by the NPSOL optimization package in 1988  \cite{day1988growth}, by the LANCELOT package in 1991 \cite{gould1991growth}, by the JuMP (Julia for Mathematical Programming) package in 2024 \cite{edelman2024some}, and, again, by LANCELOT in 2026 \cite{gould2026}.

This paper shows the exact value of $4.1325\ldots$ to have an unexpected form. We show that this number is the root of a $61$st degree polynomial with integer coefficients. We further conjecture that this lower bound is indeed the maximum (Conjecture \ref{g5conjecture}). With the benefit of hindsight, we can see that this surprising form of an answer, the root of a 61st degree polynomial, could escape prediction. While the super-exponentially many possible configurations of tight constraints prevents us from fully proving this conjecture, we prove that this value is indeed tight when restricted to the set of constraints suggested by JuMP (Theorem \ref{thm_n=5}). Our technique, involving non-linear optimization software, Grobner bases, and discriminants, provides a new approach to the growth factor problem. This approach can be readily applied to larger dimensions, and produces lower bounds for $n = 6$, $7$, and $8$, which we suppose are maximums as well. They are $5$ for $n = 6$, a root of a sixth degree polynomial for $n =7$, and $8$ for $n = 8$ (see Subsection \ref{sub:678}). In addition, using a combination of mathematical analysis and interval arithmetic, we produce an improved upper bound of $4.84$ for the maximum growth factor when $n =5$ (Theorem \ref{thm:484}), the first improvement upon Tornheim's 1968 bound of $4 \tfrac{17}{18}$ \cite{tornheim1969maximum}.\footnote{The recent mathematical improvement \cite{bisain2025new} on Wilkinson's upper bound for the growth factor under complete pivoting does not improve upon Tornheim's bound for $n=5$.}

While some experts in numerical linear algebra might imagine that this problem ought to be easier, we note that many such similar sounding problems can prove to be notoriously difficult.\footnote{For example, the number of multiplies required to compute the product of two $3\times 3$ matrices remains open \cite{heule2021new}.} The resolution of the maximum growth for $n = 5$ will likely require an impressive marriage of mathematics and computation, one that may not yet be possible in 2026.

\subsection{Gaussian elimination with complete pivoting}

Here we briefly remind the reader of the definition of Gaussian elimination with complete pivoting and the associated growth factor. For a historical account of the growth factor problem in Gaussian elimination, we refer the reader to Higham's classic text \emph{Accuracy and Stability of Numerical Algorithms} \cite[Chapter 9]{higham2002accuracy}, \cite[Section 1.1]{edelman2024some}, and \cite{urschel2025numerical}. Here and in what follows, we restrict ourselves to real $n \times n$ matrices.

Gaussian elimination is an algorithm for writing a matrix $A$ as the product of a lower triangular matrix $L$ and an upper triangular matrix $U$ by iteratively applying the following block formula:
\begin{equation}\label{ge}
A = \begin{bmatrix}
\alpha & y^T \\
x & B
\end{bmatrix} = \begin{bmatrix}1 & 0 \\ \frac{x}{\alpha} & I \end{bmatrix}  \begin{bmatrix} \alpha & y^T \\ 0 & B - \frac{xy^T}{\alpha} \end{bmatrix},
\end{equation}
where $\alpha$ is a scalar, $x$, $y$ are vectors, and $B$ is a matrix. If $B - \frac{xy^T}{\alpha} = L_1 U_1$ for a lower unitriangular $L_1$ and upper triangular $U_1$, then the LU factorization of $A$ is given by
\begin{equation}A = \begin{bmatrix}
\alpha & y^T \\
x & B
\end{bmatrix} = \begin{bmatrix}1 & 0 \\ \frac{x}{\alpha} & L_1 \end{bmatrix}  \begin{bmatrix} \alpha & y^T \\ 0 & U_1 \end{bmatrix}.\end{equation}
Let $\|\cdot\|_{\max}$ denote the largest magnitude entry of a matrix. The quantity 
 \begin{equation}g(A) = \max \left\{ \| L \|_{\max}, \frac{\|U\|_{\max}}{\norm{A}_{\max}}   \right\}, \end{equation}
is called the {\it growth factor} and appears in standard error estimates for Gaussian elimination in floating point arithmetic. The worst case stability of Gaussian elimination is governed by the size of the growth factor (a small growth factor implies numerical stability, see \cite[Theorem 9.5]{higham2002accuracy}). In order for the growth factor to be small, or even finite, pivoting of the rows or columns of the matrix may be needed. Complete pivoting is a strategy that swaps rows and columns so that, at each step, the pivot ($\alpha$ in Equation (\ref{ge})) is the largest magnitude entry of the matrix. 

Let $A_{i:j,k:\ell}$ denote the submatrix of $A$ with rows $\{i,i+1,\ldots,j\}$ and columns $\{k,k+1,\ldots,\ell\}$, and
\begin{equation}\label{GE}
A^{(k)}:= A_{k:n,k:n} - A_{k:n,1:k-1} \left[A_{1:k-1,1:k-1}\right]^{-1}A_{1:k-1,k:n} \in \mathbb{R}^{(n-k+1) \times (n-k+1)}
\end{equation}
denote the $(n-k+1) \times (n-k+1)$ matrix resulting from $k-1$ applications of Equation (\ref{ge}), with entries indexed from $k$ to $n$ (i.e., the top left entry of $A^{(k)}$ has index $k,k$). For example, $A^{(1)}$ is the original $n \times n$ matrix $A$, $A^{(2)}$ is the $(n-1)\times (n-1)$ matrix $B- x y^T / \alpha$ from Equation (\ref{ge}) resulting from one step of Gaussian elimination, and $A^{(n)}$ is the $1 \times 1$ matrix with the last pivot as its sole entry. We say that a matrix $A$ is \emph{completely pivoted} if
\begin{equation}|A^{(k)}_{i,j}| \le |A^{(k)}_{k,k}| \qquad \text{for all  } i,j = k,\ldots,n, \; k = 1,\ldots, n-1.\end{equation}
Computationally, this matrix $A$ is never explicitly formed; only the associated permutation vectors are computed, yielding a $PAQ = LU$ factorization for some permutation matrices $P$ and $Q$. However, for mathematical analysis, we may ignore the initial matrix prior to complete pivoting, and work directly with the completely pivoted one. In addition, for a completely pivoted matrix, the above definition for the growth factor reduces to
$g(A) = \max_{k \in \{1,\ldots,n\}}|U_{kk}|$
and, when considering the maximum growth factor over $n \times n$ completely pivoted matrices, it suffices to maximize $|U_{nn}|$ \cite[Lemma 5.1]{edelman2024some}.

\subsection{The algebraic complexity of the growth factor} Recall that an algebraic number is a root of a non-zero polynomial in one variable with integer coefficients, and the degree of an algebraic number is the minimum degree for which such a polynomial exists. The definition of $g(A)$, paired with the constraint structure of completely pivoted matrices, implies that the maximum growth factor over $n \times n$ completely pivoted matrices is always an algebraic number (by the Tarski-Seidenberg theorem \cite[Theorem 1.5.8]{scheiderer2024course}). We can even provide an upper bound on the degree of that number:

\begin{proposition}\label{thm:algebraic}
The largest possible growth factor for a $n \times n$ real matrix under complete pivoting is an algebraic number of degree at most $(2n)^{n^2-1}$.
\end{proposition}

\begin{proof}
Let $n>1$. Without loss of generality, we may restrict ourselves to the compact set 
\begin{equation}\{A \in \mathbb{R}^{n \times n} \, | \, \norm{A}_{\max} = A_{11} = 1\}.\end{equation} 
The maximum growth factor is non-decreasing with $n$ \cite[Lemma 5.1]{edelman2024some}, so we may consider only the maximum value of the last pivot over this set, combined with the completely pivoted constraints. In addition, without loss of generality, we may require that the pivots $A^{(k)}_{kk}$ are positive for all $k$, i.e.,
\begin{equation} \det(A_{1:k,1:k}) \ge 0 \quad \text{for} \quad k = 1,\ldots,n, \end{equation}
without affecting the maximum growth factor. Using the Schur complement, the constraint $|A_{ij}^{(k)}|\le |A_{kk}^{(k)}|$ may be equivalently represented using determinants of submatrices
\begin{equation}-\det(A_{1:k,1:k}) \le \det(A_{1:k-1 \cup \{i\},1:k-1 \cup \{j\}}) \le \det(A_{1:k,1:k}).\end{equation}
Let us add the additional variable $g$ and constraint 
\begin{equation}g \det(A_{1:n-1,1:n-1}) \le \det(A).\end{equation}
The maximum growth can be viewed as maximizing $g$ over the semi-algebraic set (a set defined by polynomial equalities and inequalities) we just defined, which has $n^2$ variables ($A_{ij}$ for $(i,j) \ne (1,1)$ and $g$) and is defined by polynomials of degree at most $n$. By \cite[Theorem 1]{jeronimo2013minimum}, the maximum value of $g$ on this compact semi-algebraic set is an algebraic number of degree at most $2^{n^2 - 1} n^{n^2-1} = (2n)^{n^2-1}$.
\end{proof}

When $n = 5$, Proposition \ref{thm:algebraic} provides an upper bound of $10^{24}$ for the degree of the solution, which appears strikingly pessimistic in comparison to our conjecture that the true maximum is the root of a $61^{st}$ degree polynomial (Conjecture \ref{g5conjecture}).

\section{A JuMP + Groebner Basis + Discriminant Polynomial Approach}\label{sec:jump}

In 1988, 1991, and 2024 \cite{day1988growth,gould1991growth,edelman2024some}, previous researchers investigated
numerical optimization as a means of finding large
growth factors.  In this work, we think of numerical
optimization as an important first step, to be followed by
algebraic polynomial techniques facilitated by
modern computation.  We emphasize that these
techniques are exact, i.e., there are no rounding
errors.  As investigating the full combinatorial explosion
of possible equality structures seems prohibitive,
we rather allow the
numerical optimization to suggest the  equality structure to investigate.
Once provided, the remaining computations are exact rigorous
mathematics. This technique, which uses numerical optimization to set boundary constraints and then performs rigorous mathematics to produce an optimum on that set, can be viewed as an analogue of the techniques introduced in \cite{edelman2024some}, which uses numerical optimization to produce candidate matrices with large numerical growth and then performs rigorous mathematics to convert them into large growth completely pivoted matrices in exact arithmetic.

Our technique is as follows. First, we perform numerical optimization with the JuMP package. Next, we fix the tight constraints produced by the numerical optimization, and leave slack matrix entries as variables. Using the equality constraints, we solve for as many of these variables as we can. This is the easy part of the process. What follows is much more difficult. Next, we use Groebner basis computations to eliminate
non-linear polynomially dependent variables. Finally, if needed, 
we use the discriminant polynomial to perform an exact optimization.

Below, we demonstrate this technique in depth for $n = 5$, and also briefly for $n = 6,7,8$. The remainder of this section is as follows. In Subsection \ref{structure5}, we provide simple Mathematica code and a detailed discussion that illustrates our technique for obtaining this $61^{st}$ degree polynomial. In Subsections \ref{sub:groeb} and \ref{sec:disc}, we give more details regarding Groebner bases and discriminant polynomials. In Subsection \ref{sub:local_max}, we convert our analysis from Subsection \ref{structure5} into a proof that this root of a $61^{st}$ degree polynomial is the maximum growth factor over all completely pivoted matrices with a certain set of equality constraints. Finally, in Subsection \ref{sub:678}, we briefly discuss the results of applying the JuMP + Groebner basis + discriminant approach to $n = 6$, $n = 7$, and $n = 8$ as well.

\subsection{Our   structure for $n=5$}
\label{structure5}
\def\RedSpade{\color{red}\spadesuit}
\def\mRedSpade{\color{red}-\spadesuit}
\def\RedDiamond{\color{red}\diamondsuit}
\def\mRedDiamond{\color{red}-\diamondsuit}
\def\RedClub{\color{red}\clubsuit}
\def\mRedClub{\color{red}-\clubsuit}

We investigate in depth
a specific elimination sequence structure
that we have seen many times in our JuMP experiments
and that leads to the growth $4.1325\dots$ reported in experiments by
\cite{day1988growth, gould1991growth}. 
At the end of this investigation we will be able to conclude
 that the reported value of
$4.1325\dots$ is exactly the root of a $61^{st}$ degree polynomial.

For  ``at a glance'' viewing, the elimination sequence of interest is portrayed pictorially
with card suits denoting
complete pivoting equality 
constraints.  We stress that these equality constraints
are not inputs to an optimization run, but rather
are outputs:
\vspace*{.1in}

\noindent
$A^{(1)},\ldots,A^{(5)} =  $ 

\noindent
\begin{minipage}{\textwidth}
\def\RedOne{\bf \color{red}1}
\def\mRedOne{\bf \color{red}-1}
\begin{equation}
\label{eq:suits}
\hspace*{.2in}
\scalebox{0.9}{$\left[
\begin{array}{rrrrr}
\RedOne & \RedOne & a_{13} & a_{14} & a_{15} \\
x & a_{22} & a_{23} & \mRedOne & \RedOne \\
\mRedOne & a_{32} & a_{33} & \mRedOne & \mRedOne \\
y & a_{42} & \RedOne & \RedOne & \RedOne \\
z & \mRedOne& \RedOne & \mRedOne & \RedOne
\end{array}
\right],
\left[
\begin{array}{rrrr}
{\color{red} {\boxed{\RedSpade} }}\hspace*{-.04in}  & \times & \times & \times \\
\RedSpade  & \times & \mRedSpade & \times \\
\times  & \RedSpade & \times & \times \\
\mRedSpade  & \times & \times & \times \\
\end{array}
\right],
\left[
\begin{array}{rrr}
{\color{red} {\fbox{$\RedDiamond$} }} \hspace*{-.04in} & \times & \mRedDiamond \\
\RedDiamond  & \times & \times  \\
\times  & \mRedDiamond & \times  \\
\end{array}
\right],
\left[
\begin{array}{rr}
{\color{red} \fbox{$\RedClub$}}\hspace*{-.04in} & \RedClub \\
\mRedClub & \RedClub
\end{array}
\right],
\left[
\begin{array}{r}
{\color{red} \bf 2   \clubsuit}
\end{array}
\right]$}
\end{equation}
\vspace{1mm}
\noindent
\end{minipage}

\begin{figure}[t]
    \fbox{\includegraphics[width=0.72\textwidth]{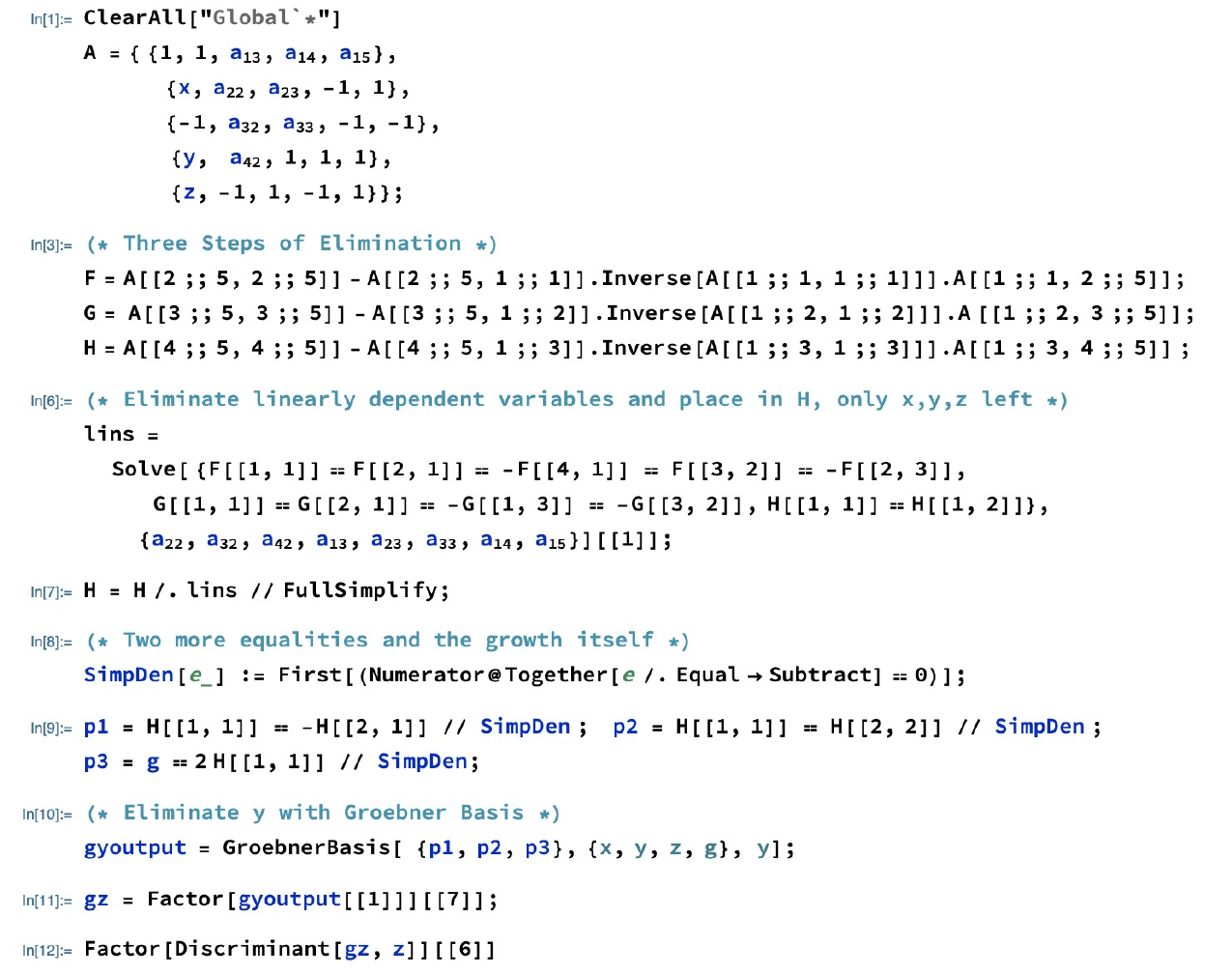}}
    \caption{The 61st degree polynomial with only 12 Mathematica input cells}
    \label{fig:math61}
\end{figure}

The Mathematica code used to compute the $61^{st}$ degree polynomial is remarkably short, and included here in Figure 1 with a screenshot and online at \cite{githubrepo}.  The brevity is a consequence, in part,
of the power of Mathematica's GroebnerBasis function. We provide an explanation of the technique, annotated with reference to the corresponding Mathematica input cells from Figure 1, referenced in {\color{blue} blue} for readability. So as to not interrupt the explanation, we provide key outputs only after the explanation.

In the above, $14$ of the $25$ entries in $A^{(1)}$
are filled in with {\color{red} \bf $\pm1 $s},
following JuMP's output. Thus we have $11$ unknowns to find. Three are labeled $x,y,z$ and eight more are labeled with the form $a_{ij}$ ({\color{blue} In[1]}). The next three steps of Gaussian Elimination
$A^{(2)},\ldots,A^{(4)}$, labeled $F,G,H$ in ({\color{blue} In[3]}), are determined exactly
by the Gaussian elimination procedure 
(see Equation (\ref{GE})).

There are ten equality
constraints that we impose on $A^{(2)}$, $A^{(3)}$, and $A^{(4)}$
based on JuMP
that are denoted by the unboxed card suits. 
They are all of the form of a $\pm$ equality with
the (boxed) upper left entry: 
$=(4\RedSpade$) + $(3\RedDiamond$) + $(3\RedClub$).
We found that we can eliminate eight of the 11 variables
(those labeled $a_{ij}$)
using 8 of these 10 equality constraints.  Specifically
we use  $(4\RedSpade$) + $(3\RedDiamond$) + $(1\RedClub$)
leaving two more $(\RedClub$)s for later  ({\color{blue} In[6]}).
The result of the elimination appears in Equations (\ref{2.2}) through (\ref{2.9}).
We then substitute Equations (\ref{2.2}) through (\ref{2.9}) into the $2 \times 2$
matrix $A^{(4)}$, $H$ in the code ({\color{blue} In[7]}).

We now have three unknowns $x,y,z$ and two equality constraints, $(2\RedClub$).  These equality constraints
are rational functions, yet we need polynomial constraints
for what follows, so we form a common denominator
({\color{blue} In[8]}), which is used when we set up
the aforementioned last two equality constraints
$(2\RedClub$) ({\color{blue} In[9]}). We also introduce
a new variable, the growth $g$ ({\color{blue} In[9]}).
These create three polynomial equalities that we
label $p_1,p_2$ and $p_3$.  These polynomials
are documented below in Equations (\ref{p1}) through (\ref{gz}). The big step is to take the polynomials
$p_1,p_2,p_3$ in $x,y,z,$ and $g$ and
compute the
Groebner basis to eliminate the variable $y$ ({\color{blue} In[10]}).  Each element of the
Groebner basis is an equality that must hold.
Upon factoring every factor is a possible equality.
We found that the $7^{th}$ one was the operative one
so we save it in the variable gz,
named because that factor only involves $g$ and $z$ ({\color{blue} In[11]}).

Finally we compute the discriminant of gz with respect to z, factor that, and take the $6^{th}$ factor, which provides
the correct $61^{st}$ degree polynomial that we are looking for ({\color{blue} In[12]}).

Next, we provide the output. The below linear constraints follow ({\color{blue} Out[6]}).
\begin{align}
&\scalebox{0.75}{$a_{22}=1 + x + z$}  \label{2.2} \\
&\scalebox{0.75}{$a_{32} =z $}\\
&\scalebox{0.75}{$a_{42}=\displaystyle\frac{-2 - 2x - 2y - 2xy - z + xz + 2x^2z + yz + 3xyz + 2x^2yz + 2z^2 + 4xz^2 + 2x^2z^2 + 2yz^2 + 3xyz^2 + z^3 + xz^3 + yz^3}{-2 - 2x + xz + 2x^2z + xz^2}$}\\
&\scalebox{0.75}{$a_{13}=\displaystyle -\frac{z}{y} $} \\
&\scalebox{0.75}{$a_{23}= \displaystyle\frac{\begin{matrix*}[l] (2y + 2xy + 2xz + 2x^2z - 2yz - 5xyz - 4x^2yz - xz^2 - 3x^2z^2 - 2x^3z^2  \\ \quad -   2yz^2 - 3xyz^2 + x^2yz^2 + 2x^3yz^2 - xz^3 - x^2z^3 + xyz^3 + 3x^2yz^3 + xyz^4) \end{matrix*}}{y(-2 - 2x + z + 3xz + 2x^2z + yz + 2xyz + z^2 + xz^2 + yz^2)}$} \\
&\scalebox{0.75}{$a_{33}= \displaystyle\frac{\begin{matrix*}[l](-2y - 2xy - 2z - 2xz + 3xyz + 2x^2yz + 2y^2z + 4xy^2z + z^2 + 3xz^2 + 2x^2z^2 + 3yz^2 \\ \quad + 3xyz^2  + 2y^2z^2 - xy^2z^2 - 2x^2y^2z^2 + z^3 + xz^3 - 2xyz^3 - y^2z^3 - 3xy^2z^3 - yz^4 - y^2z^4)\end{matrix*}}{y(-2 - 2x + z + 3xz + 2x^2z + yz + 2xyz + z^2 + xz^2 + yz^2)}$} \\
&\scalebox{0.75}{$a_{14}=-z$} \\
&\scalebox{0.75}{$a_{15}=\displaystyle \frac{z(x + z)}{1 + x} $}\label{2.9}
\end{align}
The aforementioned relationship $p_1$ arises
from the $2\times 2$ matrix $A^{(4)}$ by letting $\RedClub$ in the $(1,1)$ position equal the $-\RedClub$
in the $(2,1)$ position. This gives the first complicated polynomial relationship ({\color{blue} Out[9]}):
\begin{align}
\label{p1}
\scalebox{0.75}{$\begin{aligned}
p_1: 0 &= -4yz - 12xyz - 16x^2yz - 8x^3yz - 4y^2z - 4xy^2z - 4yz^2 - 8xyz^2 + 8x^2yz^2 + 28x^3yz^2 + 16x^4yz^2 \\
&\quad + 10y^2z^2 + 28xy^2z^2 + 28x^2y^2z^2 + 12x^3y^2z^2 + 2y^3z^2 + 4xy^3z^2 - 4z^3 - 12xz^3 - 12x^2z^3 - 4x^3z^3 \\
&\quad + 4yz^3 + 22xyz^3 + 52x^2yz^3 + 34x^3yz^3 - 8x^4yz^3 - 8x^5yz^3 + 8y^2z^3 + 9xy^2z^3 - 6x^2y^2z^3 \\
&\quad - 24x^3y^2z^3 - 16x^4y^2z^3 - 5xy^3z^3 - 2x^2y^3z^3 + 2z^4 + 12xz^4 + 26x^2z^4 + 24x^3z^4 + 8x^4z^4 \\
&\quad + 5yz^4 + 21xyz^4 + 7x^2yz^4 - 37x^3yz^4 - 32x^4yz^4 - 4x^5yz^4 - 4y^2z^4 - 13xy^2z^4 - 30x^2y^2z^4 \\
&\quad - 27x^3y^2z^4 + 4x^4y^2z^4 + 4x^5y^2z^4 - 3y^3z^4 - 2xy^3z^4 + 2x^2y^3z^4 + 2z^5 + 7xz^5 + 4x^2z^5 - 9x^3z^5 \\
&\quad - 12x^4z^5 - 4x^5z^5 - 9xyz^5 - 36x^2yz^5 - 31x^3yz^5 + 4x^5yz^5 - 2y^2z^5 - 9xy^2z^5 - 14x^2y^2z^5 \\
&\quad + 10x^3y^2z^5 + 12x^4y^2z^5 + 3xy^3z^5 - 2xz^6 - 8x^2z^6 - 10x^3z^6 - 4x^4z^6 - yz^6 - 8xyz^6 - 6x^2yz^6 \\
&\quad + 9x^3yz^6 + 8x^4yz^6 - 2xy^2z^6 + 8x^2y^2z^6 + 13x^3y^2z^6 + y^3z^6 - xz^7 - 2x^2z^7 - x^3z^7 + xyz^7 \\
&\quad + 6x^2yz^7 + 5x^3yz^7 + 2xy^2z^7 + 6x^2y^2z^7 + xyz^8 + x^2yz^8 + xy^2z^8
\end{aligned}$}
\end{align}
The aforementioned relationship $p_2$ arises
from the $2\times 2$ matrix $A^{(4)}$ by letting $\RedClub$ in the $(1,1)$ position equal the $\RedClub$
in the $(2,2)$ position.  This gives the other complicated polynomial relationship ({\color{blue} Out[9]}):
\begin{align}
\label{p2}
\scalebox{0.75}{$\begin{aligned}
p2: 0 &= -4xyz - 8x^2yz - 4x^3yz + 4y^2z + 12xy^2z + 12x^2y^2z + 4x^3y^2z - 4z^2 - 12xz^2 - 12x^2z^2 - 4x^3z^2 \\
&\quad - 5yz^2 - 5xyz^2 + 12x^2yz^2 + 20x^3yz^2 + 8x^4yz^2 - 2y^2z^2 - 8xy^2z^2 - 14x^2y^2z^2 - 16x^3y^2z^2 \\
&\quad - 8x^4y^2z^2 - y^3z^2 - 3xy^3z^2 - 2x^2y^3z^2 + 2z^3 + 12xz^3 + 26x^2z^3 + 24x^3z^3 + 8x^4z^3 \\
&\quad + 6yz^3 + 20xyz^3 + 17x^2yz^3 - 9x^3yz^3 - 16x^4yz^3 - 4x^5yz^3 - 2xy^2z^3 - 14x^2y^2z^3 \\
&\quad - 11x^3y^2z^3 + 4x^4y^2z^3 + 4x^5y^2z^3 + 2xy^3z^3 + 2x^2y^3z^3 + 2z^4 + 7xz^4 + 4x^2z^4 - 9x^3z^4 \\
&\quad - 12x^4z^4 - 4x^5z^4 + yz^4 - 6xyz^4 - 24x^2yz^4 - 21x^3yz^4 + 4x^5yz^4 - 8xy^2z^4 - 6x^2y^2z^4 \\
&\quad + 10x^3y^2z^4 + 12x^4y^2z^4 + y^3z^4 + xy^3z^4 - 2xz^5 - 8x^2z^5 - 10x^3z^5 - 4x^4z^5 - 2yz^5 \\
&\quad - 7xyz^5 - 4x^2yz^5 + 9x^3yz^5 + 8x^4yz^5 - 2y^2z^5 - xy^2z^5 + 8x^2y^2z^5 + 13x^3y^2z^5 \\
&\quad - xz^6 - 2x^2z^6 - x^3z^6 + xyz^6 + 6x^2yz^6 + 5x^3yz^6 + 2xy^2z^6 + 6x^2y^2z^6 + xyz^7 + x^2yz^7 + xy^2z^7
\end{aligned}$}
\end{align}
Finally there is the introduction of the growth factor $g$ that gives the equality 
$g={\color{red} \bf 2 \  \clubsuit}$. This may be rewritten as yet another 
polynomial expression ({\color{blue} Out[9]}):
\vspace{-.2in}
\begin{center}
\begin{equation}
\label{gz}
\scalebox{0.75}{$
p_3: 0=2 g x^2 z+g x z^2+g x z-2 g x-2 g-8 x^2 z-4 x z^2-4 x z+8 x-2 y z^2+2 y z-2 z^2+2 z+8.$}
\end{equation}
\end{center}

As described above, we compute a Groebner basis of the above three polynomials in
$g$ and $x,y,z$ ({\color{blue} Out[10]}).
Any polynomial in the basis has the property 
that it must equal $0$ if the input polynomials are equal to $0$.
The first (as computed by Mathematica) of the polynomials in that basis is
\begin{equation}
\label{groutput}
\scalebox{0.75}{$
\begin{aligned}
&(g-2)^3 (z-1)^2 z^3 (g z+2 g-6 z-8)\,\bigl(g+z^2-z-4\bigr)^2
 \bigl(2 x^2 z+x z^2+x z-2 x-2\bigr)^2 \\
&\times \Bigl(
 g^7 z^{10}-36 g^6 z^{10}+548 g^5 z^{10}-4560 g^4 z^{10}+22304 g^3 z^{10}-63744 g^2 z^{10}+97792 g z^{10}-61440 z^{10} \\
&\quad -5 g^7 z^{9}+176 g^6 z^{9}-2628 g^5 z^{9}+21472 g^4 z^{9}-103008 g^3 z^{9}+287744 g^2 z^{9}-429056 g z^{9}+261120 z^{9} \\
&\quad 2 g^7 z^{8}-88 g^6 z^{8}+1640 g^5 z^{8}-16512 g^4 z^{8}+96384 g^3 z^{8}-325376 g^2 z^{8}+589824 g z^{8}-450560 z^{8} \\
&\quad 22 g^7 z^{7}-664 g^6 z^{7}+8168 g^5 z^{7}-51488 g^4 z^{7}+167424 g^3 z^{7}-218112 g^2 z^{7}-116736 g z^{7}+423936 z^{7} \\
&\quad -19 g^7 z^{6}+636 g^6 z^{6}-8972 g^5 z^{6}+67920 g^4 z^{6}-291744 g^3 z^{6}+690432 g^2 z^{6}-793088 g z^{6}+299008 z^{6} \\
&\quad -41 g^7 z^{5}+1024 g^6 z^{5}-9588 g^5 z^{5}+36288 g^4 z^{5}+7520 g^3 z^{5}-484864 g^2 z^{5}+1393664 g z^{5}-1291264 z^{5} \\
&\quad 32 g^7 z^{4}-944 g^6 z^{4}+11424 g^5 z^{4}-71424 g^4 z^{4}+237312 g^3 z^{4}-372992 g^2 z^{4}+134144 g z^{4}+172032 z^{4} \\
&\quad 40 g^7 z^{3}-792 g^6 z^{3}+4688 g^5 z^{3}+5760 g^4 z^{3}-184064 g^3 z^{3}+840192 g^2 z^{3}-1642496 g z^{3}+1220608 z^{3} \\
&\quad -16 g^7 z^{2}+464 g^6 z^{2}-5408 g^5 z^{2}+32256 g^4 z^{2}-103680 g^3 z^{2}+171008 g^2 z^{2}-118784 g z^{2}+16384 z^{2} \\
&\quad -16 g^7 z+224 g^6 z-17152 g^4 z+131072 g^3 z-442368 g^2 z+720896 g z-458752 z \\
&\quad -64 g^6+1408 g^5-12800 g^4+61440 g^3-163840 g^2+97792 g+229376 g-131072
\Bigr) \\
\end{aligned}
$}
\end{equation}

Taking the discriminant of the final big factor with respect to $z$ allows us to
perform a maximization, as described in Section \ref{sec:disc}. The computed discriminant ({\color{blue} Out[12]}) is 
\begin{equation}
\scalebox{0.75}{$-295147905179352825856 \, (-6 + g)\, (-4 + g)^{32}\, (-2 + g)^6 \,(40 - 12 g + g^2) P_5(g),$}
\end{equation}
where $P_5(g)$ is a $61^{st}$ degree polynomial
\begin{equation}\label{eq:61}
\hspace{-.05in}\scalebox{0.67}{$
\begin{aligned}
P_{5}(g) =&-400768351683901408958416200638054732473753927680  +4485865416851125743788230145233574388718469382144 g 
\\ &\quad
-24514764764411504983475981567602588918013382098944 g^2 
+87217226729799564487654554570961532637620826275840 g^3 
\\ &\quad
 -227272096308101830763401243729746552751035059798016 g^4 
 +462709819431332579135613503090210881152929786494976 g^5 
 \\ &\quad
 -766717614581304313114392121166832616680006210289664 g^6 
 +1063584436778076241182112335678621890599851466424320 g^7 \\
&\quad -1260890928707415271320049672620311457803946132766720 g^8 
 +1297738577606655379676433338360545479999727889022976 g^9 
 \\ &\quad
-1174034757320339655369158795901049591606791293108224 g^{10} 
 +942963464577089526044422206269395994977092205805568 g^{11} \\
&\quad -677934188506657333601022620611924661223810387673088 g^{12} 
 +439257245648213960218425406267346647338070336077824 g^{13} 
 \\ &\quad
 -257972856502979631073226298004753211978646395289600 g^{14} 
 +137990608561232610866439277823535118550017691877376 g^{15} \\
&\quad -67501782245716434625321022717049489901181623861248 g^{16} 
 +30300948936890380444641188810690357840398421852160 g^{17} 
 \\ &\quad
 -12517004528640474114410312100200354935951234957312 g^{18} 
 +4769063904740641993461152447863884474835999391744 g^{19} \\
&\quad -1678801195187379807996886928950958962405971853312 g^{20} 
+546626688596819879421936148562972655095260381184 g^{21} 
\\ &\quad
 -164708488081737671005260911304389856200929312768 g^{22} 
 +45913441353285356544812597189623621359713124352 g^{23} \\
&\quad -11824394908442237813782015042429910964838072320 g^{24} 
+2805505141981336865142529528869875207720927232 g^{25} 
\\ &\quad
-610122837879497813508978054510894441606152192 g^{26} 
+120517644367496052573488222809065538013102080 g^{27} \\
&\quad -21259613889629901444312810746680889153945600 g^{28} 
 +3232473902649233225492038672001558428778496 g^{29} 
\\ &\quad
-385658581119287455693748758468032752254976 g^{30} 
+22806308021126342831010091967753643622400 g^{31} \\
&\quad +4894535361089103056466100216447232901120 g^{32} 
-2321369271980741328279078542236324986880 g^{33} 
\\ &\quad
+601696561599723882176326969201384226816 g^{34} 
 -123775438994232337878919408695562469376 g^{35} \\
&\quad 
+ 21961720969887997312065520856381521920 g^{36} -3468454832875402800794241505433747456 g^{37} \\ &\quad
 +494585628535060280610858619094695936 g^{38} 
-64096987281967089485798296793382912 g^{39} \\ &\quad
+7566621187666403129807306257072128 g^{40} 
-812991033858742708150732650971136 g^{41} 
 +79218174622382297258263303946240 g^{42} \\ &\quad
-6951821319653222073914715275264 g^{43} 
 +542828623292663115763310526464 g^{44} 
 -36902811774200980940618924032 g^{45}\\ &\quad
+2088035191940616804066754560 g^{46} 
-86850177265000967745748992 g^{47}  
+1176129227591098999291904 g^{48} \\ &\quad
+223725965850750547116032 g^{49} 
-29310655078216507848704 g^{50} 
 +2327920657178572930048 g^{51} 
-143068110048321554432 g^{52} \\ &\quad
+7211069563930583552 g^{53} 
-302975406829565568 g^{54} +10602327706308288 g^{55} 
-305386806375072 g^{56} 
+7079217962544 g^{57}  \\
&\quad
-127283265864 g^{58} 
+1668783060 g^{59} 
-14211126 g^{60} 
+59049 g^{61}
\end{aligned}
$}
\end{equation}

Note that the roots of the first four terms cannot achieve the maximum growth factor, as $g = 2,4,6$ cannot be maximal and $g^2 -12 g +40$ does not have real roots. There is only one root of the $61^{st}$ degree polynomial $P_5(g)$ that is real
and has an absolute value between $4$ and $5$.  This root is
the relevant growth factor.  Here, it is listed to 200 decimal places  

\vspace{.1in}
\noindent
$g=$
\vspace{-.05in}
{\tiny{
4.1325170786 3247285422 3346853277 3712699527 9153779908 7694544180 5219674378 3429681764 8258551799 8548302263 \\
 \hspace*{.35in}
0152515659  8878252716 0955420993 4703653072 2238973129 2506908826 4378769458 9154507591 7179407904 4884836275
}}
\vspace{-.4in}
\begin{equation}
\label{eq:200}
\end{equation}
Nick Gould reported to us that Lancelot in quad precision gave $4.132517078632472854223346853277$ \cite{gould2026}, which agrees with our exact number to 30 digits.

We first computed the real roots of $P_5(g)$ numerically, but we also applied
Descartes' rule of signs to confirm rigorously that there was only
one real root in $(4,5)$ and no real roots in $(-5,0)$.
The first computation was performed by replacing $g$ with $(5t+4)/(t+1)$
which maps $0$ to $4$ and $\infty$ to $5$.  There is then one sign
change to the coefficients of the polynomial: they are negative up to $t^9$ and
positive from $t^{10}$ onward.
Similarly, replacing $g$ with $-5/(t+1)$ gives a polynomial with all negative
coefficients, indicating no real roots in $(-5,0)$.

To a few digits, the matrix that achieves this growth is
$$
\begin{bmatrix}
1.0 & 1.0 & 0.581691 & -0.453225 & -0.194706 \\
-0.617533 & 0.835692 & -0.997327 & -1.0 & 1.0 \\
-1.0 & 0.453225 & 0.854664 & -1.0 & -1.0 \\
-0.779151 & 0.635656 & 1.0 & 1.0 & 1.0 \\
0.453225 & -1.0 & 1.0 & -1.0 & 1.0
\end{bmatrix}
$$
and the first three steps of Gaussian elimination are
$$\scalebox{0.85}{$
\left[
\begin{array}{rrrr}
1.45322 & -0.638114 & -1.27988 & 0.879763 \\
1.45322 & 1.43635 & -1.45322 & -1.19471 \\
1.41481 & 1.45322 & 0.646869 & 0.848295 \\
-1.45322 & 0.736363 & -0.794587 & 1.08825
\end{array}
\right]
\left[
\begin{array}{rrr}
2.07447 & -0.173344 & -2.07447 \\
2.07447 & 1.89292 & -0.00821026 \\
0.0982495 & -2.07447 & 1.96801
\end{array}
\right]
\left[
\begin{array}{rr}
2.06626 & 2.06626 \\
-2.06626 & 2.06626
\end{array}
\right],$}
$$
finally giving $4.13251659953668686\dots$. The above $5\times5$ matrix turns out to be the matrix JuMP consistently output when tasked with numerical optimization over $5\times 5$ matrix space. The confluence of numerics and algebraics leads us to 

\begin{conjecture}
\label{g5conjecture}
    The largest possible growth factor for a $5\times 5$ real matrix under complete pivoting is exactly the lower bound proven in this paper, i.e., the  unique real root $g$ in the interval $[4,5]$ of the $61^{st}$ degree polynomial $P_5(g)$ in Equation (\ref{eq:61}).
\end{conjecture}

Finally, to fully specify the matrix, we need to specify $x,y$ and $z$. These are all also roots of $61^{st}$ degree polynomials. They were obtained by further manipulations of the GroebnerBasis command in Mathematica. 

The polynomials are
the ``x'' polynomial:

\vspace{0.05in}
\noindent
\scalebox{0.4}{$\begin{aligned}
&108808296205965 + 1323194540232146 x + 7865647933595184 x^2 + 31228825556949182 x^3 + 107997522640726035 x^4 + 393008440268830416 x^5 + 1348746754816041054 x^6 + 3516523333118803362 x^7 + 5754391756232757792 x^8 + 5214192347219981126 x^9 \\
&+ 25559515264418742281 x^{10} + 257756069644412767390 x^{11} + 1478476929917814560735 x^{12} + 5944350544466261317336 x^{13} + 18841484398850183767876 x^{14} + 50067185659426716770294 x^{15} + 115768365699336480978069 x^{16} \\
&+ 238811520215407160405020 x^{17} + 446976057406517194517281 x^{18} + 766681291506732029731084 x^{19} + 1209478466277603557811900 x^{20} + 1752034274842248363736252 x^{21} + 2320103033750789484598264 x^{22} \\
&+ 2795767846772084986875856 x^{23} + 3059790836237140855776448 x^{24} + 3052792360230287271763848 x^{25} + 2812526824619623805828396 x^{26} + 2452431226607390790409200 x^{27} + 2091804113502123089776816 x^{28} \\
&+ 1790219701110818693279744 x^{29} + 1534589964580390467054400 x^{30} + 1279758600557167276110624 x^{31} + 1000220301935479853207840 x^{32} + 710865838849672772321024 x^{33} + 450150680810695414827680 x^{34} \\
&+ 250216080730382637830656 x^{35} + 119963441884896234008896 x^{36} + 47734966126253307693440 x^{37} + 13748464696449687666496 x^{38} + 521283858117700221696 x^{39} - 3230898705726340566144 x^{40} - 3371166500901150918656 x^{41} \\
&- 2484008961481633683456 x^{42} - 1551209750210795104256 x^{43} - 862001911474277894144 x^{44} - 433907584308191352832 x^{45} - 199328329712629332992 x^{46} - 83924485174467334144 x^{47} - 32569558262570657792 x^{48} \\
&- 11763116963647160320 x^{49} - 4000336330575282176 x^{50} - 1288020887487414272 x^{51} - 388595263218909184 x^{52} - 106818885310480384 x^{53} - 25736311788535808 x^{54} - 5205637428936704 x^{55} - 841810651381760 x^{56} \\
&- 101761280901120 x^{57} - 8103697317888 x^{58} - 276907950080 x^{59} + 13555990528 x^{60} + 1342177280 x^{61}
\end{aligned}$}

\vspace{0.1in}
the ``y'' polynomial:

\vspace{0.05in}
\noindent
\scalebox{0.4}{$\begin{aligned}
&24100440737218560 + 614707117564852224 y + 7409397905960361984 y^2 + 55335781800370578240 y^3 + 297757114018948548288 y^4 + 1293646296791996625888 y^5 + 4610187311578114584960 y^6 + 13019703705952918498248 y^7 \\
&+ 27842809842297210580632 y^8 + 47406057900797091505768 y^9 + 92104944568890787283540 y^{10} + 238839746461916970817831 y^{11} + 603731402589652570590433 y^{12} + 1034576259342255535202382 y^{13} \\
&+ 849699439930067575964688 y^{14} + 1087120300631917167801093 y^{15} + 5337232504276996897407233 y^{16} + 9341147077215643094414853 y^{17} + 2150484036636145649699589 y^{18} - 1539706427646045020171185 y^{19} \\
&+ 17347899649663175808914197 y^{20} + 17993820321067413748649257 y^{21} - 19859619062759266049328581 y^{22} - 34423040704405594292030535 y^{23} - 14952003461190607537754597 y^{24} \\
&+ 3544562845758706109084628 y^{25} + 22173683920054147384680566 y^{26} + 26186639305431858851376694 y^{27} + 12562810886769633600216862 y^{28} + 3169052942979959644884648 y^{29} \\
&- 4930946179734364548724116 y^{30} - 6892078031353485209467410 y^{31} - 1550842268915919394938666 y^{32} - 648927173612250605474346 y^{33} + 684621102957720618449534 y^{34} + 824194874798279989919762 y^{35} \\
&- 127135702414578921921010 y^{36} + 250336340224205802864718 y^{37} - 125201124031225616814118 y^{38} - 40910300204617121838506 y^{39} + 29577504846979094411450 y^{40} - 33393326478241128181556 y^{41} \\
&+ 19969543071878938354596 y^{42} - 6497025583248224521049 y^{43} + 2147602211195475968945 y^{44} - 243130574811470985046 y^{45} - 86457949373884105780 y^{46} + 18360567359353845597 y^{47} \\
&+ 1439232788089387529 y^{48} - 7387602599581293171 y^{49} + 4970428743971171285 y^{50} - 1618245509195375249 y^{51} + 279795470991728013 y^{52} - 23449866473147039 y^{53} + 354830570976547 y^{54} \\
&+ 32370567985881 y^{55} + 8653830126003 y^{56} - 1021857916784 y^{57} - 12840301254 y^{58} + 2900498924 y^{59} - 60012576 y^{60} + 3240 y^{61}
\end{aligned}$}

\vspace{0.1in}
and the ``z'' polynomial:
\vspace{0.05in}

\noindent
\scalebox{0.4}{$\begin{aligned}
&92274688 - 77594624 z - 3563061248 z^2 + 12612796416 z^3 + 84373667840 z^4 - 342328475648 z^5 - 1188624203776 z^6 + 4500312522752 z^7 + 9965798555648 z^8 - 35298216222720 z^9 - 52106858799104 z^{10} \\
&+ 179009813813248 z^{11} + 181400708323328 z^{12} - 619081507733504 z^{13} - 457606108923904 z^{14} + 1555107328205824 z^{15} + 875429080727552 z^{16} - 3009793252430272 z^{17} - 1183306308902592 z^{18} \\
&+ 4498516732667648 z^{19} + 1071817419160480 z^{20} - 5161759476215664 z^{21} - 711222485207280 z^{22} + 4804858762334852 z^{23} + 236137035604908 z^{24} - 3713628474267088 z^{25} + 208012386142114 z^{26} \\
&+ 2376617552796062 z^{27} - 470046308577526 z^{28} - 1239850195372619 z^{29} + 536319096846224 z^{30} + 462681956120773 z^{31} - 434803022616248 z^{32} - 40752001774936 z^{33} + 242716156384854 z^{34} \\
&- 92075638942988 z^{35} - 88110341587590 z^{36} + 82366635479207 z^{37} + 10426374501948 z^{38} - 38467108766581 z^{39} + 9139981364064 z^{40} + 9174500210418 z^{41} - 4816222590562 z^{42} \\
&- 1025017379346 z^{43} + 1026060362262 z^{44} + 251143343171 z^{45} - 368086401960 z^{46} + 61967139675 z^{47} + 44034784936 z^{48} - 19349434492 z^{49} - 968529318 z^{50} + 2044222048 z^{51} \\
&- 179315122 z^{52} - 150295551 z^{53} + 26632332 z^{54} + 10129681 z^{55} - 3312444 z^{56} - 156558 z^{57} + 155408 z^{58} - 1056 z^{59} - 6216 z^{60} + 720 z^{61}
\end{aligned}$} .

\subsection{The role of the Groebner basis}\label{sub:groeb}

The Groebner basis computation takes as inputs
a sequence of polynomials in several variables  and, if successful, outputs another  sequence of polynomials 
with at least one variable eliminated.  The key property is that if all the input polynomials are $0$ at its arguments, then so is every output polynomial.

In our computation for $n = 5$ we had the three polynomials $p_1,p_2,p_3$ in the four variables $x,y,z$ and $g$.
The command ({\color{blue} {In[10]}}) in our Mathematica code
asks for the elimination of $y$.  There are several output polynomials but we listed the first one in
Equation (\ref{groutput}) as it is key for our computation.
That polynomial depends on $x,g$ and $z$ but not $y$.
The elimination of $y$ has been used to describe
Groebner bases as a generalization of Gaussian elimination
to polynomials.

If we factor any polynomial in the output sequence
then one of the factors is necessarily zero when the input polynomials are zero.

Groebner bases have many other properties and
applications.  One good starting reference is \cite{sturmfels2005grobner}.

\subsection{The role of the discriminant polynomial}
\label{sec:disc}

Suppose we define $g(z)$ implicitly through a polynomial equation 
$P(g,z)=0,$ and we wish to find a candidate local maximum of $g$
as a function of the parameter $z$.  
We can differentiate $P(g(z),z)=0$ to obtain
$$P_g  \frac{dg}{dz}  + P_z =0$$ which implies
$$ \frac{dg}{dz} = - \frac{P_z }{ P_g}.$$
Therefore a stationary point satisfies
$$ P(g,z)=0, \ {\rm and} \ P_z(g,z)=0,$$
as long as $ P_g(g,z) \ne 0.$

A well known method for solving $P=P_z=0$ is to think of $P(g,z)$
as a polynomial in $z$ with coefficients depending on $g$.
We are then asking for which $g$ does this polynomial have a double root.
We recall that the  discriminant polynomial
Disc$_z P$ is a polynomial in $g$ whose roots provide 
the possible $g$ for which $P=P_z=0.$
In Mathematica, this polynomial may be readily computed with the command
\verb+Discriminant[P, z]+ and thus a potential maximum growth factor $g$
is obtained as a root of a univariate polynmomial in $g$.

\subsection{ 
The growth value 4.1325... is optimal
 for  Sequence (\ref{eq:suits})
}\label{sub:local_max}

\begin{theorem}\label{thm_n=5}
Over all completely pivoted matrices that satisfy the equality constraints portrayed in (\ref{eq:suits}), the largest growth factor is obtained exactly as the
root of $P_5(g)$ that is approximately
$4.1325\dots$.
\end{theorem}

\begin{proof}
We recall that at least one of the factors in the
discriminant (\ref{groutput}) must be $0$.
We must, one by one, consider the other factors and show them to be extraneous.  We provide a Mathematica code that goes through the computations in \cite{githubrepo}, and here we report
on the methodology and results.

\noindent
\textbf{Case 1:} $(g-2)$: $g=2$  is lower than the maximum so it is readily dismissed.

\noindent
\textbf{Case 2:}  $(z-1)$: If we let $z\rightarrow 1$, we obtain
the top left entry of $H=A^{(4)}$ is 2 leading to a growth of \phantom{.}\hspace{1.2 cm} at most $4<4.1325\ldots$.

\noindent
\textbf{Case 3:}  $z$:  If we let $z\rightarrow 0$, we  again obtain
the top left entry of $H=A^{(4)}$ is 2, as in Case 2.

\noindent
\textbf{Case 4:} $(gz+2g-6z-8)$:  We asked Mathematica to compute the Groebner basis with $p_1,p_2,p_3,$ and \phantom{.}\hspace{1.2 cm} $(gz+2g-6z-8)$.
The first element of the result is
$$
z^2 (z+1)    
(xz+z^2+z-1)
   (2 x^2 z+x z^2+x z-2 x-2)^2
   $$  
   $$
   \ \  \  \times
   \left(z^8+11 z^7+38 z^6+54 z^5-3 z^4-55 z^3-14 z^2+44 z+8\right) .
$$
\phantom{.}\hspace{1.3 cm}So now we must examine each of these five possibilities,
one by one, as well.

\noindent\phantom{.}\hspace{1.3 cm}\textbf{Case 4a:} $z$: $z=0$ was already dismissed.

\noindent\phantom{.}\hspace{1.3 cm}\textbf{Case 4b:} $(z+1)$: $z=1$ forces $y$ to be $0$ or $1$.
If $y=0$ then $A_{1,3}$ blows up, and if $y=1$, then \noindent\phantom{.}\hspace{1.3 cm}$H_{11}=2$ which has already been ruled out.

\noindent\phantom{.}\hspace{1.3 cm}\textbf{Case 4c:} $(xz+z^2+z-1)$: This can be rewritten as
$x=\frac{1-z-z^2}{z}$ which forces $A_{2,2}$ to be $1/z$ \noindent\phantom{.}\hspace{1.3 cm}hence $z=-1$ (as $z=1$ is ruled out), but this forces $x=1$
which in turn blows up $A_{2,2}$.

\noindent\phantom{.}\hspace{1.3 cm}\textbf{Case 4d:} $(2x^2z+xz^2+xz-2x-2)$: adding this to the Groebner basis forces either $y=-1$, \noindent\phantom{.}\hspace{1.3 cm} $z=1$, or $z=0$, all already ruled out.

\noindent\phantom{.}\hspace{1.3 cm}\textbf{Case 4e:} $(z^8+11 z^7+38 z^6+54 z^5-3 z^4-55 z^3-14 z^2+44 z+8)$: Of the eight roots, only \\ \noindent\phantom{.}\hspace{1.2 cm} one is in $[-1,1]$,
$z = -0.17852433207456908\ldots$.  Case 4 requires, $g= \frac{2(4+3z)}{2+z}$, which is a \\\noindent\phantom{.}\hspace{1.3 cm} monotonically increasing function, and $z$ must be bigger than $0.14$ to give a better maximum \\ \noindent\phantom{.}\hspace{1.3 cm} than the one we have, so this is also ruled out.

\noindent
\textbf{Case 5:} $(2x^2z+xz^2+xz-2x-2)$:
Adding this to the Groebner Basis only allows $y$ to be $-1$ or $z$ to \noindent\phantom{.}\hspace{1.3 cm}be $0$ or $1$, all already ruled out.
\end{proof}

\subsection{What about $n =6$, $7$, and $8$?}\label{sub:678}

In 1988, Day and Peterson numerically found growth $5$ for $n = 6$ and growth $6$ for $n = 7$ \cite{day1988growth} with NPSOL, in 1991 Gould matched these numbers with LANCELOT \cite{gould1991growth}. In 2024, we also found growth $5$ for $n = 6$, but found a larger growth of $6.05\dots$ for $n =7$ \cite{edelman2024some}. The JuMP+GroebnerBasis+Discriminant method can also be readily applied to these situations, and produces $5$ for $n = 6$ and a root of a polynomial of degree $6$
$$P_7(g)=3g^6 - 86g^5 + 688g^4 - 1136g^3 - 5888g^2 + 12032g + 22528$$
for $n = 7$. This polynomial has only one real root with absolute value between $3$ and $16$, namely
$$g =6.056953473721059619788033246920631605921538842353885034320248936556863912133036\dots.$$
See \cite{githubrepo} for the associated Mathematica files for $n = 6$ and $n = 7$. Interestingly, the method is notably simpler in dimensions $n =6$ and $n = 7$, but for two different reasons. When $n =5$, the matrix is fairly complicated and there are only $23$ equality constraints, one too few to define the $24$ variables of a matrix with first pivot $1$. When $n = 6$, the JuMP computed matrix has a very simple form, with only $6$ entries not equal to $\pm 1$. For this reason, no Groebner basis is needed and the discriminant approach is overkill for the resulting optimization problem 
$$\mathrm{maximize} \quad |x^2 - 5| \quad \text{subject to} \quad x \in [-1,1].$$
When $n = 7$, the JuMP computed matrix is more complicated, but also has at least $n^2-1$ equality constraints, removing the need for any optimization, after the computation of a Groebner basis. Finally, we note that when $n = 8$, numerical optimization suggests growth equal to $8$, a value known since the 1960's, and achieved by Sylvester's $8 \times 8$ Hadamard matrix. However, interestingly, in 2024, we found non-Hadamard $8\times 8$ matrices that also achieve growth $8$ \cite[Table 1]{edelman2024some}.

If our lower bounds turn out to be the true
maximum growth factors, the first $8$ largest growth factors are indeed
a strange mathematical sequence: $1,2,2 \frac{1}{4},4,$
(a root of a $61^{st}$ degree polynomial), $5$, (a root of a $6^{th}$ degree polynomial), and $8$.

\section{Upper bounds on the growth of $5 \times 5$ matrices with complete pivoting}

Using mathematical analysis combined with interval arithmetic in the form of branch-and-bound algorithms, we improve the upper bound for the growth factor of $5 \times 5$ completely pivoted matrices. We remind the reader that interval arithmetic is an \textit{exact} numerical computation in that computed results are an interval such that the true answer in the absence of rounding is contained in that interval.

Let $p_k = |A_{kk}^{(k)}|$ be the $k^{th}$ pivot in Gaussian elimination. Recall that the maximum growth factor over $n \times n$ completely pivoted matrices is equal to the maximum $n^{th}$ pivot $p_n$ over $n \times n$ completely pivoted matrices $A$ with $A_{11} = 1$. Therefore, in what follows, we only concern ourselves with maximizing the fifth pivot $p_5$ over  completely pivoted normalized matrices (defined in Subsection \ref{sub:3x3}). Given a  completely pivoted normalized $5 \times 5$ matrix $A$, $p_5$ is equal to the product of the $3^{rd}$ pivot $p_3$ (or alternatively, the last pivot of the top left $3 \times 3$) and the $3^{rd}$ pivot of the normalized bottom right $3 \times 3$ after two iterations of Gaussian elimination of $A$. Combined with the fact that the growth of $3 \times 3$ matrices is at most $2 \frac 14$, this gives an upper bound of $(2\frac 14)^2 = 5\frac{1}{16}$ for the growth. This bound has been lowered to $5.005$ in \cite{cohen1974note} and $4\frac{17}{18} = 4.9\bar{4}$ in \cite{tornheim1969maximum} (which happened to be five years earlier).

Still, the gap between the current best known upper bound for the $5 \times 5$ case and our lower bound of $4.1325...$ remains significant. We make mild progress on this gap, reducing the upper bound slightly, to $4.84$ (Theorem \ref{thm:484}). There is nothing inherently special about this value; it can be readily lowered by simply lowering a handful of numbers in the associated code and suffering a slightly longer runtime. The main contribution here is the presentation of new \emph{ideas} for addressing this problem. However, the presented technique, even when taken to its natural limit, certainly falls short of $4.1325...$. To fully close this gap, further innovations are required.

\begin{table}[t]
    \centering
    \caption{Bounds for $n = 5$}
    \label{tab:growth}
    \begin{tabular}{|l|l|}
        \hline
        $n =5 $ growth & reference   \\  \hline      
        $5.0625 = (2 \frac 14 )^2$ \phantom{\huge I} & well-known \\
        $5.005$ & \cite[1974]{cohen1974note} \\
        $4.9\bar{4}$ & \cite[1969]{tornheim1969maximum} \\
        $4.84$ & Theorem \ref{thm:484} \\
        $4.1325\dots$  & Conjecture \ref{g5conjecture} \\
        \hline
    \end{tabular}
\end{table}

\subsection{High level technique} 
As noted above, we can combine bounds for the top $3\times 3$  and the bottom $3\times 3$ 
parts of a $5\times 5$ elimination to produce 
an upper bound of $5 \frac{1}{16}$ for $p_5$.
Our technique is an improvement of this idea in two respects. 

We might summarize at a glance what follows 
by 
drawing the reader's attention to Figure \ref{fig:s1s2}, Figure \ref{fig:p3p3p}, and Inequality (\ref{eq:low_rank_5x5_pivot}).
First, using mathematical analysis, we produce a (necessary, but not sufficient) test to determine if a given $3\times 3$ matrix with large growth could occur as $A^{(3)}$ of a completely pivoted normalized $ 5 \times 5$ matrix with large third pivot (Inequality (\ref{eq:low_rank_5x5_pivot})). In essence, we have created a test that shows the incompatibility of large growth $3 \times 3$ matrices with each other, so that we may conclude that they never occur together as $A^{(3)}$ and the top left $3 \times 3$ of $A$, respectively. Then, using interval arithmetic and branch-and-bound techniques, we obtain a description of what $3 \times 3$ matrices with large growth can look like (Figure \ref{fig:s1s2}), so that we may apply our test to them. Finally, we implement this mathematical test using exact computation (Figure \ref{fig:p3p3p}), which allows us to rule out the existence of certain large growth $5 \times 5$ matrices (Theorem \ref{thm:484}).

\subsection{A test for $3 \times 3$ matrices}\label{sub:3x3}
The growth factor of a matrix $A$ is invariant under rescaling i.e. $g(A) = g(\lambda A)$ for any $\lambda \neq 0$. In addition, the growth factor is invariant under multiplying rows or columns by $-1$. \color{black} Hence, when studying the growth factor of $n \times n$ matrices, by rescaling and multiplying rows and columns by $-1$, we may assume that $\norm{A}_{\max} = 1$, $A_{11} = 1$ and $A_{i1}, A_{1i} \geq 0$ for $i=2,...,n$. We define such matrices to be \textit{normalized}.

The following lemma provides a necessary condition for a $5 \times 5$ matrix to have growth factor $g$ by reducing to a problem involving the approximation of certain normalized completely pivoted $3 \times 3$ matrices by certain rank $2$ matrices.
\begin{lemma}
\label{lemma:reduction_3x3}
    Let $A$ be a completely pivoted $5 \times 5$ matrix with $A_{11} = 1$ and $\norm{A}_{\max} = 1$ and let $p_i$ be the pivots of $A$. Then there exists a completely pivoted matrix $C \in \R^{3 \times 3}$ and vectors $u,v,x,y \in \R^3$ such that 
\begin{align}
\label{eq:low_rank_5x5_pivot} 
    &\norm{p_3 C - uv^T - p_2 xy^T}_{\max} \leq 1\\
    p_3 C = A^{(3)}, \qquad &\norm{C}_{\max} = 1,\qquad \norm{u}_\infty, \norm{v}_\infty, \norm{x}_\infty, \norm{y}_\infty \leq 1 \nonumber.
\end{align}
Moreover, if we denote $p_3'$ to be the $3$rd pivot of $C$, then $p_3 p_3' = p_5 = g(A)$.
\end{lemma}

\begin{proof}
    We may perform two iterations of Gaussian elimination on $A$ to obtain a representation of the bottom right $3 \times 3$ after two steps. Here, $a,b,c$ are scalars, $u,v,x',y'$ are vectors in $\R^3$ and $M$ is a matrix in $\R^{3 \times 3}$:
    \begin{align}
        A &= \begin{bmatrix}
            1 & a & v^T\\
            b & c & y'^T \\
            u & x' & M\\
        \end{bmatrix} \\ 
        A^{(2)} &= \begin{bmatrix}
            c-ab & y'^T - bv^T\\
            x'-au & M - uv^T\\
        \end{bmatrix} &=:
        \begin{bmatrix}
            \pm p_2 & p_2 y^T\\
            p_2 x & M - uv^T\\
        \end{bmatrix} \\
        A^{(3)} &= 
            M - uv^T \mp p_2 xy^T
         &=: p_3 C
    \end{align}
    where we defined $x = p_2^{-1}(x'-au)$ and $y = p_2^{-1}(y'-bv)$. Since $A$ is completely pivoted, so are $A^{(2)}$ and $A^{(3)}$ and thus we have $\norm{x}_\infty, \norm{y}_\infty \leq 1$.

    In the above, we have defined $C$ via $A^{(3)} = p_3 C$ where $p_3$ is the third pivot and $C$ is a completely pivoted $3 \times 3$ matrix with $\norm{C}_{\max} = 1$ so that $g(A) = p_3 g(C)$.

    Since $\norm{A}_{\max} = 1$, we have $\norm{u}_\infty, \norm{v}_\infty, \norm{M}_{\max} \leq 1$. Rearranging, we have 
    \begin{equation}
        \norm{p_3 C + uv^T \pm p_2 xy^T}_{\max} \leq 1.
    \end{equation}
    We may multiply $u$ and $x$ by the appropriate signs (which doesn't change their norms) to get the result.
\end{proof}

Without loss of generality, $C$ can be taken to be normalized by multiplying rows and columns of $A$ by $-1$ (recall that multiplying rows or columns of $A$ does not change the growth factor). Thus, we can bound $g(A)$ by showing that if $g(C) = p_3'$ and $p_3$ are large, then Inequality \eqref{eq:low_rank_5x5_pivot} cannot hold.

\subsection{Construction of covers}\label{sub:cover}
Here, we compute approximations \textit{from above} of the set of  completely pivoted normalized $3 \times 3$ matrices $C$ whose growth factor is sufficiently large. More precisely, we build sets $S_{2.2}$ and $S_{2.15}$ in the form of unions of dyadic boxes constructed using interval arithmetic, such that $S_{2.2}$ contains all  completely pivoted normalized $3 \times 3$ matrices $C$ with $g(C) \geq 2.2$, with a similar statement for $S_{2.15}$.

\begin{remark}
The choice of the pair of numbers $2.15$ and $2.2$ are mostly arbitrary, based on our tolerance for computational run time. The key property of this pair is that $2.2^2 = 4.84$, our desired upper bound, and $2.15$ is such that $2.15 \times 2.25 = (2.2 - .05)(2.2 + .05) < 2.2^2 = 4.84$.
\end{remark}

    We construct the sets $S_{2.2}$ and $S_{1.5}$ out of boxes $B$ of the form 
    \begin{equation}
        B = \begin{bmatrix}
            1 & I_{12}^+ & I_{13}^+ \\
            I_{21}^+ & I_{22} & I_{23} \\
            I_{31}^+ & I_{32} & I_{33}
        \end{bmatrix}
    \end{equation}
    where the $I_{ij}$ and $I_{ij}^+$ are dyadic intervals of the form 
    \begin{align}
        I_{ij} \in  \set{[k2^{-4},(k+1)2^{-4}]:k\in\Z\cap[-2^4,2^4-1]} \\ 
        I_{ij}^+ \in  \set{[k2^{-4},(k+1)2^{-4}]:k\in\Z\cap[0,2^4-1]}
    \end{align}
    Thus, each box $B$ is a subset of $\R^{3 \times 3}$ and has side lengths $\frac{1}{16}$ in all coordinates except in the degenerate case corresponding to $B_{11}$. The total number of such boxes $B$ can be calculated as $2^{4 \times 4 + 4 \times 5} = 2^{36} \approx 6.87 \times 10^{10}$.

    We now formally state tne properties of the sets $S_{2.2}$ and $S_{2.15}$ computed.
\begin{lemma}[Pruning boxes]
\label{lemma:box_covering_1}
    There exists explicitly computable subsets $S_{2.2} \subset \R^{3 \times 3}$ and $S_{2.15} \subset \R^{3 \times 3}$ such that 
    \begin{itemize}
        \item $S_{2.2}$ and $S_{2.2}$ are a union of $1882$ and $19511$ disjoint boxes (of the form $B$) respectively.
        \item Every matrix in $S_{2.2}$ and $S_{2.15}$ is normalized and $S_{2.2} \subset S_{2.15}$. 
        \item $S_{2.2}$ contains all completely pivoted normalized matrices $C$ with $g(C) \geq 2.2$ and $S_{2.15}$ contains all completely pivoted normalized matrices $C$ with $g(C) \geq 2.15$.
    \end{itemize}

\end{lemma}

We note that in principle one can refine the side lengths of the boxes $B$ to be $2^{-l}$ for $l > 4$ to obtain better approximations of $\mathcal{C}_{2.2}$ and $\mathcal{C}_{2.15}$. However this is expensive computationally due to the high dimensionality of the problem. The numbers $1882$ and $19511$ are not special, only that they are significantly smaller than the total number of such boxes $2^{36}$ and small enough for computation to be feasible.

The proof involves computing the sets $S_{2.2}$ and $S_{2.15}$ explicitly. In Figure \ref{fig:s1s2} we plot two-dimensional slices of them by projecting onto two entries of the matrices at a time. We note that the sets $S_{2.2}$ and $S_{2.15}$ appear to cluster around $3$ particular matrices in $\R^{3\times3}$. These matrices are marked in black dots in Figure \ref{fig:s1s2} and they correspond exactly to the three completely pivoted normalized $3 \times 3$ matrices that achieve the maximum growth of $2.25$. We list them below for completeness:

\begin{equation*}
    \begin{bmatrix}
1 & 1 & 0.5 \\
1 & -0.5 & -1 \\
0.5 & -1 & 1 \\
\end{bmatrix}, \qquad
\begin{bmatrix}
1 & 1 & 0.5 \\
0.5 & -1 & 1 \\
1 & -0.5 & -1 \\
\end{bmatrix}, \qquad
\begin{bmatrix}
1 & 0.5 & 1 \\
1 & -1 & -0.5 \\
0.5 & 1 & -1 \\
\end{bmatrix}.
\end{equation*}

\begin{figure}[t]
    \centering
    \subfloat[]{\includegraphics[width = 0.333\textwidth]{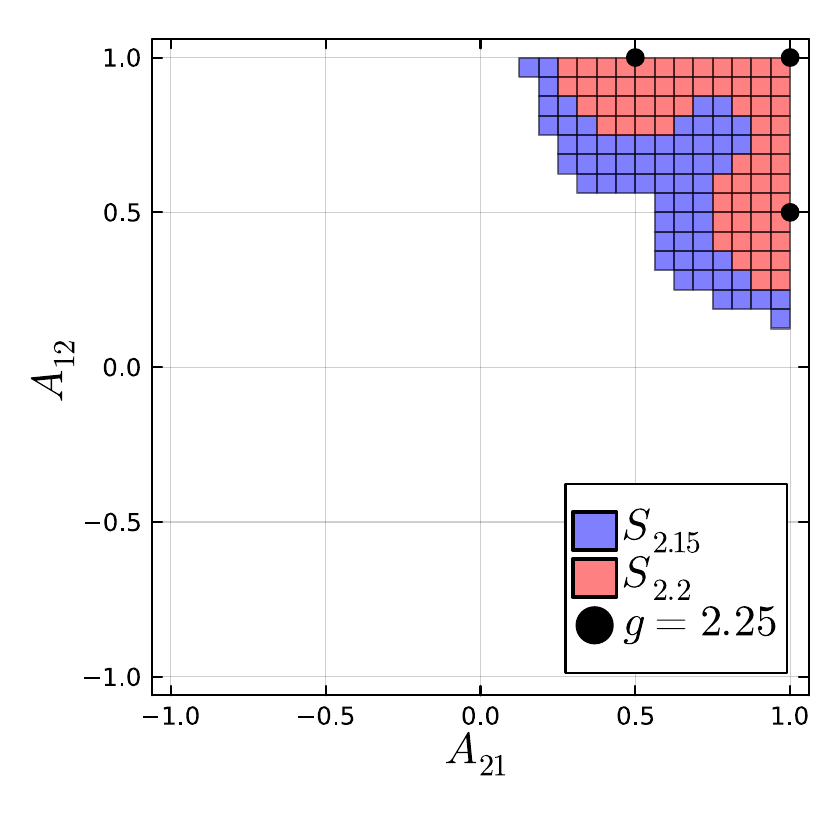}}
    \subfloat[]{\includegraphics[width = 0.333\textwidth]{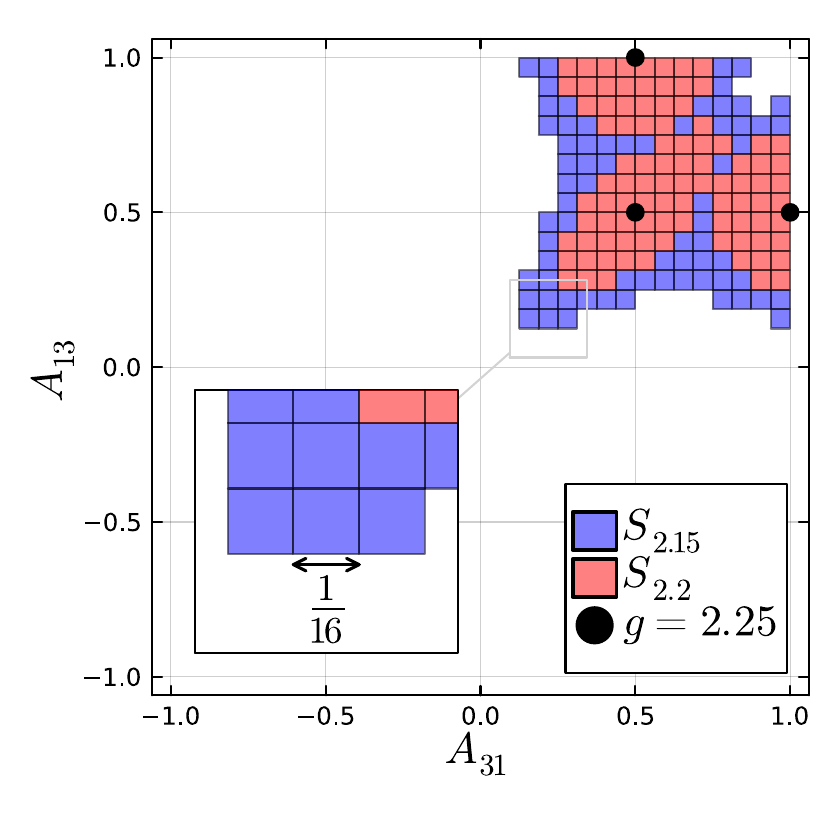}}
    \subfloat[]{\includegraphics[width = 0.333\textwidth]{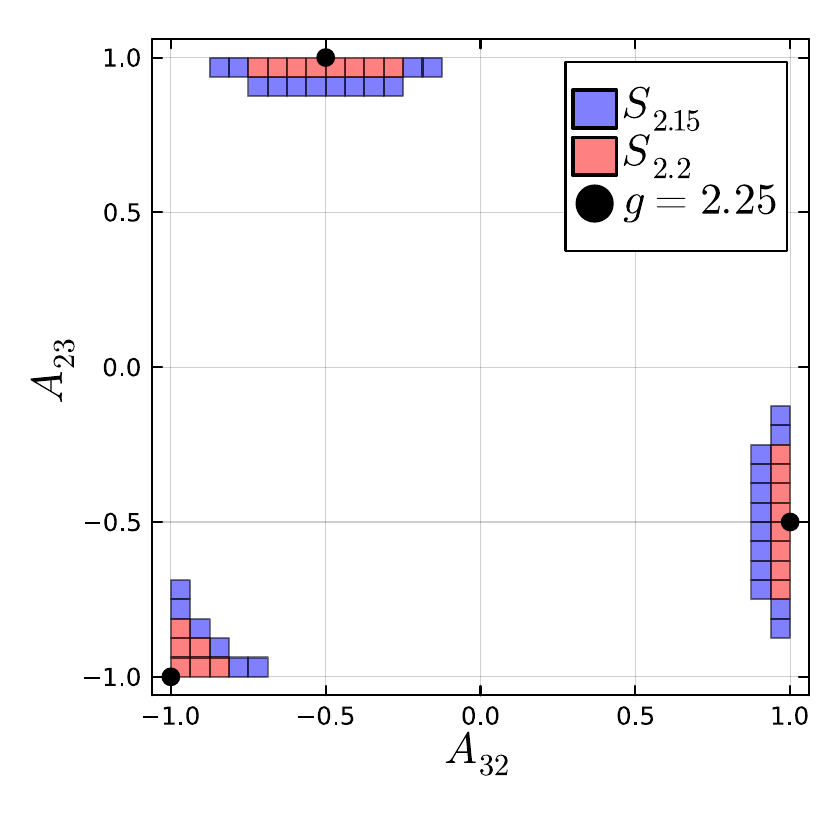}}
    \caption{2D Projections of $S_{2.2}$ and $S_{2.15}$ onto entries $A_{21}$ and $A_{12}$ (left), entries $A_{31}$ and $A_{13}$ (middle), and entries $A_{32}$ and $A_{23}$ (right). Note that $S_{2.2} \subset S_{2.15}$ and thus every box plotted that is part of $S_{2.2}$ is also part of $S_{2.15}$.} 
    \label{fig:s1s2}
\end{figure}

\begin{proof}
    \textit{Computer assisted with interval arithmetic}. The key component of the computation involves a routine \texttt{check\_growth}$(B, g)$ whose inputs are a box $B \subset \R^{3 \times 3}$ (not necessarily with side lengths $2^{-4}$) and a growth factor tolerance $g > 2$. This routine returns \texttt{true} whenever $B$ is \textit{guaranteed} to contain no completely pivoted normalized matrices $C$ with $g(C) = p_3' \geq g$, and \texttt{false} when the routine cannot guarantee this (though it may still not contain any such matrices).
    
    Let $C$ be a completely pivoted normalized $3 \times 3$ matrix and $p_2', p_3'$ be its pivots (recall that $g(C) = p_3'$). We quote the following facts from \cite{cohen1974note} which are used in the routine:
    \begin{itemize}
        \item The maximum growth factor of a $2 \times 2$ matrix is $2$ and as such, for a $3 \times 3$ matrix to have growth factor $p_3' = g$, we must have $p_2' \geq \frac g2$. As such, if $g > 2$ then we must have $C_{22}^{(1)} \leq -\frac g2$, where $C_{22}^{(1)} = C_{22} - C_{21} C_{12}$ i.e. the $2,2$ entry after 1 iteration of Gaussian elimination.
        \item If $p_3' > 2$ then the sign pattern of the bottom right $2 \times 2$ after one iteration of Gaussian elimination must be one of the following 
        \begin{equation*}
            \begin{bmatrix}
                * & * & * \\
                * & - & + \\
                * & - & - \\
            \end{bmatrix},
            \qquad
            \begin{bmatrix}
                * & * & * \\
                * & - & - \\
                * & + & - \\
            \end{bmatrix},
            \qquad
            \begin{bmatrix}
                * & * & * \\
                * & - & - \\
                * & - & + \\
            \end{bmatrix}.
        \end{equation*}
        \item Defining $\alpha_1$ and $\alpha_2$ as in Theorem 1.1 of \cite{cohen1974note} (whose values depend on the above sign patterns), we have
        \begin{equation*}
            p_3' \leq \frac 14 \alpha_1 \alpha_2 \left(
                1 + \frac{1}{\alpha_1} + \frac{1}{\alpha_2}
            \right)^2.
        \end{equation*} 
        We also have that $p_3' \leq 3p_2' - p_2'^2$.
    \end{itemize}
    Using these above facts, we can check when these facts are violated using interval arithmetic to determine whether $B$ contains \textit{no} completely pivoted normalized $3 \times 3$ matrices with growth $p_3' \geq g$ for some prespecified $g>2$.

    The rest of the computation involves a standard branch-and-bound procedure to refine an initial box $B_0 = \set{C \in \R^{3 \times 3} : C_{11} = 1, \norm{C}_{\max} = 1, C_{12}, C_{13}, C_{21}, C_{31} \geq 0}$ to side lengths $\frac{1}{16}$, pruning away any box which \texttt{check\_growth} determines to contain no high growth matrices. The output in the case $g=2.2$ is a set $S_{2.2}$ consisting of $1882$ boxes and in the case $g=2.15$ is a set $S_{2.15}$ consisting of $19511$ boxes. The programs and their details for execution can be found at \cite{githubrepo} and were executed using 64 threads in approximately $12$ hours on Julia 1.10.4..
\end{proof}

\subsection{Infeasibility of Inequality \eqref{eq:low_rank_5x5_pivot}} \label{sub:infeas}
The sets $S_{2.2}$ and $S_{2.15}$ from Lemma \ref{lemma:box_covering_1} enclose the set of completely pivoted normalized $3 \times 3$ matrices. Our goal is now to show that given a completely pivoted $5 \times 5$ matrix $A$, if the pivots $p_2$ and $p_3$ take certain values (depending on whether $g=2.2$ or $g=2.15$) then Inequality \eqref{eq:low_rank_5x5_pivot} cannot be satisfied for $C \in S_g$ and thus $p_3 C$ cannot be the matrix obtained after two steps of Gaussian elimination from a completely pivoted matrix $A$. This is summarized in the following lemma:

\begin{lemma}
\label{lemma:nonsatisfiable_2}
    For any matrix completely pivoted normalized matrix $C$ with $p_3' \geq 2.2$ and $p_3 = 2.15$, $p_2 = \frac 32 + \sqrt{\frac 94 - p_3}$, we have
    \begin{equation}
    \label{eq:bound_nonsatisfiable}
        \norm{p_3 C - uv^T - p_2 xy^T}_{\max} > 1
    \end{equation}
    for any $u,v,x,y \in \R^3$ such that $\norm{u}_\infty, \norm{v}_\infty, \norm{x}_\infty, \norm{y}_\infty \leq 1$.

    Similarly, the above statement also holds for any completely pivoted normalized matrix $C$ with $p_3' \geq 2.15$ and $p_3 = 2.2$, $p_2 = \frac 32 + \sqrt{\frac 94 - p_3}$.
\end{lemma}
\begin{proof}
    \textit{Computer assisted with interval arithmetic}. By Lemma \ref{lemma:box_covering_1}  it is sufficient to prove the statement for all $C \in S_{2.2}$ and $C \in S_{2.15}$.
    
    Fix a box $B$ in $S_{2.2}$ or $S_{2.15}$. We define a routine \texttt{calc\_range} which given positive scalars $p_2$, $p_3$, a box $B$ and interval ranges for $u,v,y,x$, computes an interval containing $\min_{C \in B} \norm{p_3 C - uv^T - p_2 xy^T}_{\max}$.
    The routine then returns \texttt{true} if the lower bound of the interval is strictly greater than $1$.
    
    The computation follows a standard branch-and-bound procedure and verifies that for all $1882$ boxes in $S_{2.2}$ and $p_3 = 2.15$ that Equation \eqref{eq:bound_nonsatisfiable} holds, and similarly for all $19511$ boxes in $S_{2.15}$ and $p_3 = 2.2$.
\end{proof}

We now move on to the proof of the improved bound for the maximum growth of $5 \times 5$ completely pivoted matrices. 
\begin{theorem}\label{thm:484}
    For all completely pivoted $5 \times 5$ matrices, $g(A) \leq 4.84$.
\end{theorem}
\begin{proof}
    Let $A$ be a completely pivoted $5 \times 5$ real matrix. Suppose for contradiction $g(A) > 4.84$. Then by denoting $p_3$ to be the $3$rd pivot of $A$ and $p_3'$ to be the $3$rd pivot of $C$ where $A^{(3)} = p_3 C$, we know that $g(A) = p_3 p_3'$. Since $p_3, p_3' \leq 2.25$, we must have $p_3 > 2.15$ and $p_3' > 2.15$ otherwise $p_3 p_3' \leq 4.84$.

    Note that from \cite{cohen1974note}, we know that for any completely pivoted $3 \times 3$ matrix with pivots $p_2$ and $p_3$, we must have 
    \begin{equation}
        \frac 32 - \sqrt{\frac 94 - p_3} \leq p_2 \leq \frac 32 + \sqrt{\frac 94 - p_3} .
    \end{equation}

    Now suppose that $p_3 > 2.15$ and $p_3' \geq 2.2$ so that $C \in S_{2.2}$. Then, by Lemma \ref{lemma:reduction_3x3}, Inequality \eqref{eq:low_rank_5x5_pivot} holds for some value of $p_3 \geq 2.15$ and $p_2 \leq \frac 32 + \sqrt{\frac 94 - p_3}$.

    Since $\frac{2.15}{p_3} \leq 1$ and $\frac{p_2}{ \frac 32 + \sqrt{\frac 94 - p_3}}\leq 1$ we can rescale $u$ and $x$ to obtain that Equation \ref{eq:low_rank_5x5_pivot} holds for the case $p_3 = 2.15$ described in Lemma \ref{lemma:nonsatisfiable_2}, a contradiction. A similar statement holds if we assume instead that $p_3 > 2.2$ and $p_3' \geq 2.15$ so that $C \in S_{2.15}$. We thus obtain two cases since the case $p_3 < 2.15$ was already ruled out (Figure \ref{fig:p3p3p}).
    \begin{itemize}
        \item If $p_3 \in [2.2, 2.25]$ then $p_3' < 2.15$
        \item If $p_3 \in [2.15, 2.2]$ then $p_3' < 2.2$
    \end{itemize}
    Maximizing the product $p_3 p_3'$ over both cases, we obtain $p_5 \leq 4.84$, contradicting our original assumption.
\end{proof}

\begin{figure}[t]
    \centering
    \includegraphics[width=0.5\linewidth]{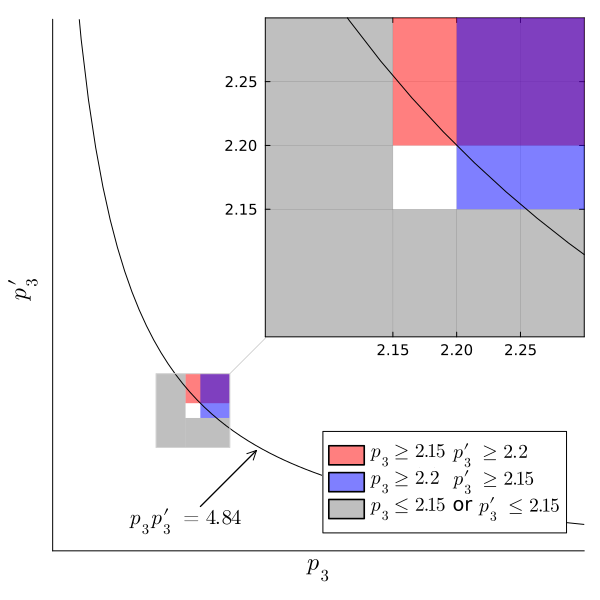}
    \caption{Under the assumption that $p_3 p_3' > 4.84$, the gray section is ruled out trivially. Computations involving $S_{2.2}$ and $S_{2.15}$ rule out the red and blue regions respectively.}
    \label{fig:p3p3p}
\end{figure}

\section*{Acknowledgements}

\setstretch{.5}
\begin{onehalfspacing}

{\tiny
The authors would like to thank Nick Gould and an anonymous referee for their comments, which have significantly improved the manuscript. This material is based upon work supported by the U.S. National Science Foundation under award Nos DMS-2513687, CNS-2346520,  RISE-2425761, and DMS-2325184, by the Defense Advanced Research Projects Agency (DARPA) under Agreement No. HR00112490488,  by the Department of Energy, National Nuclear Security Administration under Award Number DE-NA0003965 and by the United States Air Force Research Laboratory under Cooperative Agreement Number FA8750-19-2-1000.  Neither the United States Government nor any agency thereof, nor any of their employees, makes any warranty, express or implied, or assumes any legal liability or responsibility for the accuracy, completeness, or usefulness of any information, apparatus, product, or process disclosed, or represents that its use would not infringe privately owned rights. Reference herein to any specific commercial product, process, or service by trade name, trademark, manufacturer, or otherwise does not necessarily constitute or imply its endorsement, recommendation, or favoring by the United States Government or any agency thereof. The views and opinions of authors expressed herein do not necessarily state or reflect those of the United States Government or any agency thereof." The views and conclusions contained in this document are those of the authors and should not be interpreted as representing the official policies, either expressed or implied, of the United States Air Force or the U.S. Government. 
}
\end{onehalfspacing}

{ \small 
	\bibliographystyle{plain}
	\bibliography{main.bib} }

\end{document}